\documentclass[12pt]{article}  
\pdfoutput=1 	
\usepackage{geometry}                		
\usepackage[parfill]{parskip}    		
\usepackage{graphicx}				
\usepackage[hyperindex=true,pagebackref,colorlinks=true,pdfpagemode=none,urlcolor=blue,linkcolor=blue,citecolor=blue,pdfstartview=FitH]{hyperref}
\usepackage{amsmath,amsfonts}
\usepackage{color}
\usepackage{amsthm}
\usepackage{boolexpr}
\usepackage{amssymb,latexsym,pstricks,pst-plot}
\usepackage[english]{babel}
\usepackage{pgf}
\usepackage{tikz}
\usetikzlibrary{arrows,automata}
\usepackage[latin1]{inputenc}
\usepackage[UKenglish]{isodate}
\usepackage[T1]{fontenc}
\usepackage{listings}
\usepackage{xcolor}
\usepackage{hyperref}
\usepackage{float}

\allowdisplaybreaks

\newcommand{\seqnum}[1]{\href{https://oeis.org/#1}{\rm \underline{#1}}}

\usepackage{amssymb}
\newtheorem{theorem}{Theorem}[section]
\newtheorem{lemma}[theorem]{Lemma}

\newtheorem{corollary}[theorem]{Corollary}

\theoremstyle{definition}

\title{The appearance function for paper-folding words}
\author{Rob Burns}

\begin{document}
\maketitle
\begin{abstract}
We provide a complete characterisation of the appearance function for paper-folding sequences for factors of any length. We make use of the software package {\tt Walnut} to establish these results.
\end{abstract}

\section{Introduction}
\label{intro}

The regular paper-folding sequence begins \mbox{ 1, 1, -1, 1, 1, -1, -1, 1, \dots}. It is derived from the hills and valleys created when a piece of paper is folded length-wise multiple times. In the limit, it consists of an infinite sequence of $1$'s and $-1$'s, in which a $1$ corresponds to a hill and a $-1$ to a valley in the unfolded paper. This sequence, with $-1$'s replaced by $0$'s,  appears as sequence \seqnum{A014577} in the {\it On-Line Encyclopedia of Integer Sequences} (OEIS).\cite{oeis} In a more general form, introduced by Davis and Knuth,\cite{davis1970} a paper-folding sequence is derived from an infinite folding instruction set. This set of instructions also consists of an infinite sequence of $1$'s and $-1$'s and determines the way in which the paper is folded. The regular paper-folding sequence is produced by an instruction set consisting of an infinite sequence of $1$'s. 

In this paper, we study the appearance function of the paper-folding sequences. We show that, when $n \geq 7$, the appearance function is determined in a simple way by the folding instruction set. The connection between the folding instructions and the appearance function for smaller values of $n$ is not quite as simple, but can still be described.

We make use of the software package {\tt Walnut}. Hamoon Mousavi, who wrote the program, has provided an introductory article \cite{Mousavi:2016aa}. Papers that have used {\tt Walnut} include \cite{Mousavi2016DecisionAF}, \cite{https://doi.org/10.48550/arxiv.2103.10904}, \cite{https://doi.org/10.48550/arxiv.2110.06244}, \cite{Go__2013}, \cite{DU2017146}. Further resources related to {\tt Walnut} can be found at \href{https://www.cs.uwaterloo.ca/~shallit/papers.html}{Jeffrey Shallit's page}.

The free open-source mathematics software system SageMath \cite{sagemath} was used to check some of the calculations.

\bigskip

\section{Background and notation}
\label{bg}

Let $w$ be an infinite word. The first element of $w$ is denoted by $w[1]$, the second element by $w[2]$ etc.. A sub-word of $w$ is a continuous set of elements contained within $w$. The sub-word $w[i], w[i+1], w[i+2], \dots , w[j]$ will be abbreviated to $w[i:j]$. A finite word is called a factor of $w$ if it appears as a sub-word of $w$. The length $k$ prefix of $w$ is the length $k$ sub-word of $w$ starting with the element $w[1]$, i.e. the subword $w[1 : k]$. 

The appearance function of $w$, denoted $A_w(n)$, is defined to be the least integer k such that a copy of each length $n$ factor of $w$ is contained in the prefix $w[1 : k]$. For convenience we will use a related function which we call $S_w(n)$. $S_w(n)$ is defined to be the least integer k such that a copy of each length $n$ factor of $w$ \textbf{starts} somewhere within the prefix $w[1 : k]$. The two functions are connected by the equation $A_w(n) = S_w(n) + n - 1$.

The set of folding instructions associated with a paper-folding sequence will be denoted by $f = ( f_0, f_1, f_2, \dots )$. The paper-folding sequence associated to the folding instructions $f$ will be denoted by $P_f = P_f[1], P_f[2], \dots$. The appearance function of $P_f$ will be abbreviated to $A_f$ and $S_{P_f}$ will be abbreviated to $S_f$.

Dekking, Mend\'{e}s France and van der Poorten \cite{dekking1982} showed that the value of $P_f[k]$ can be written in terms of $f$ in a fairly simple way. If $k = 2^s \cdot r$ where $r$ is odd, then

\begin{equation}
\label{pfn}
P_f [k] =
\begin{cases}
   f_s, & \text{if } \, \, r \equiv 1 \pmod{4}\\
   - f_s, & \text{if } \, \, r \equiv 3 \pmod{4}
\end{cases}
\end{equation}

Schaeffer \cite{lukes}  observed that equation (\ref{pfn}) leads to a 5-state deterministic finite automaton that takes, as input the base-2 expansion of an integer $k$ in parallel with the folding instructions $f$ and outputs $P_f [k]$. The automaton outputs the correct value of $P_f[k]$ provided that enough folding instructions have been read in. In particular, more than $\log_2 (k)$ folding instructions must be included in the input. The {\tt Walnut} software package includes this automaton and can be used to investigate its behaviour.

For integers $n$, define the function $\phi (n)$ to be the least integer $r$ such that $r \geq n$ and $r$ is a power of $2$. So, if $2^{k-1} < n \leq 2^k$, then $\phi(n) = 2^k$. An alternative definition is that 
$$
\phi(n) = 2^k, \text{   where   } k = \lceil \log_2(n) \rceil.
$$

Schaeffer \cite{lukes} showed that, when $n \geq 3$:
\begin{equation}
\label{sfmax}
\max_{f} S_f (n) = 6 \cdot \phi(n) .
\end{equation}
Go{\v{c}} et al. \cite{Go__2015} showed that when $n \geq 7$,
\begin{equation}
\label{sfmin}
\min_{f} S_f (n) = 4 \cdot \phi(n) .
\end{equation}

\bigskip

\section{Formula for $S_f$ and $A_f$}

We begin this section with some {\tt Walnut} commands. {\tt Walnut} includes an automaton which takes as parallel input a folding instruction set $f$ and the base-2 representation of an integer $k$, written in least significant digit (lsd) first order, and outputs the value of $P_f[k]$. The {\tt Walnut} representation of $P_f [k]$ is $PF [f] [k]$. We now introduce some {\tt Walnut} formulae which will be useful. Firstly we define the automaton {\tt pffaceq} which takes as parallel input the folding instruction set $f$ and three integers $i$, $j$ and $n$, written in base-2 lsd format. The resulting automaton accepts the input if and only if the length $n$ subwords of $P_f$ starting at indices $i$ and $j$ are identical. Since it has 153 states, it cannot be displayed here.
\begin{verbatim}
def pffaceq  "?lsd_2 Ak (k < n) => PF[f][i+k] = PF[f][j+k]":
\end{verbatim}

The following code creates an automaton related to the function $\phi$ which was defined in section~\ref{bg}. The automaton takes two integers $x$ and $y$ as input, written in base-2 lsd form. It accepts the input if $x$ is a power of $2$ and $\phi(y) = x$. We use the name {\tt pfphi} for this automaton. It is pictured in figure~\ref{pfphifig}.

\begin{verbatim}
reg power2 lsd_2 "0*10*":
def pfphi "?lsd_2 $power2(x) & (x >= y) & x < 2*y":
\end{verbatim}

\begin{figure}[H]
\begin{center}
    \includegraphics[width=6in]{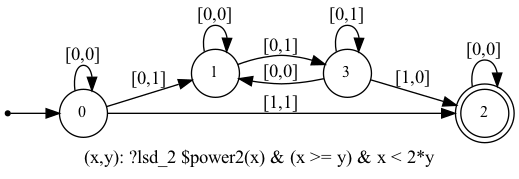}
    \end{center}
    \caption{Automaton {\tt pfphi}.}
    \label{pfphifig}
\end{figure}

We now create an automaton {\tt pfapp} which takes as parallel input a folding instruction set $f$ and two integers $i$ and $n$, written in base-2 lsd format. The resulting automaton accepts the input if and only if the length $n$ factor of $P_f$ starting at index $i$ does not appear earlier within $P_f$. This automaton has $121$ states.
\begin{verbatim}
def pfapp "?lsd_2 (Aj (j<i) => (Et t<n & PF[f][i+t] != PF[f][j+t]))":
\end{verbatim}

\bigskip

We next establish some preliminary results.

\bigskip

\begin{lemma}
\label{pfapp}
When $n \geq 7$, the length $n$ factor $P_f[6\cdot \phi(n) : 6\cdot \phi(n) + n -1]$ first appears in $P_f$ starting at either index $4\cdot \phi(n)$ or index $6\cdot \phi(n)$. This factor appears nowhere else in the prefix $P_f[1 : 6\cdot \phi(n) + n -1]$ of $P_f$.
\end{lemma}
\begin{proof}
We create an automaton which takes as input the folding instructions $f$ and an integer $n$ written in base-2 lsd form. It accepts the input if there is an index $k < 6\cdot \phi(n)$ such that $k \neq 4\cdot \phi(n)$ and the two factors $P_f[6\cdot \phi(n) : 6\cdot \phi(n) + n -1]$ and $P_f[k : k + n -1]$ are identical.
\begin{verbatim}
eval pftemp "?lsd_2 (n >= 7) & (Ex, k x >= 1 & $pfphi(x,n) & 
   (k != 4*x) & (k<6*x) & (Ai (i<n) => PF[f][k+i]=PF[f][6*x+i]))":
\end{verbatim}
The automaton accepts no input showing that the factor $P_f[6\cdot \phi(n) : 6\cdot \phi(n) + n -1]$ first appears either at index $4\cdot \phi(n)$ or $6\cdot \phi(n)$ and appears nowhere else in the prefix $P_f[1 : 6\cdot \phi(n) + n -1]$.
\end{proof}

\bigskip

\begin{lemma}
\label{last}
When $n \geq 7$, the factor $P_f[6\cdot \phi(n) : 6\cdot \phi(n) + n -1]$ is always the last length $n$ factor to appear in $P_f$.
\end{lemma}
\begin{proof}
We know from (\ref{sfmax}) that no length $n$ factor of $P_f$ can first begin later than index $6\cdot \phi(n)$. We create an automaton which takes as parallel input a folding instruction set $f$ and two integers $i$ and $n$, written in base-2 lsd format. The automaton accepts the input if and only if the length $n$ factors beginning at $4\cdot \phi(n)$ and $6\cdot \phi(n)$ are identical, $4\cdot \phi(n) < i < 6\cdot \phi(n)$ and the factor $P_f[i : i + n - 1]$ appears no earlier within $P_f$.
\begin{verbatim}
eval pftemp "?lsd_2 (Ex (x >= 1) & $pfphi(x,n) & 
   $pffaceq(f,4*x,6*x,n) & (Ei (i>4*x) & (i<6*x) & $pfapp(f,i,n)))":
\end{verbatim}
The automaton accepts no input. So, if the factor $P_f[6\cdot \phi(n) : 6\cdot \phi(n) + n - 1]$ first appears starting at $6\cdot \phi(n)$ then it is the last factor to appear in $P_f$ because of (\ref{sfmax}). If, on the other hand, it first appears starting at index $4\cdot \phi(n)$ (the only other possibility due to lemma \ref{pfapp}), it is again the last factor to appear, otherwise {\tt pftemp} would accept some input.
\end{proof}

\bigskip

\begin{lemma}
\label{same}
Let $n \geq 7$ with $2^{k-1} < n \leq 2^k$, so $\phi(n) = 2^k$. Then the factors $P_f[6\cdot 2^k : 6\cdot 2^k + n -1]$ and $P_f[6\cdot 2^k : 6\cdot 2^k + 2^k - 1]$ first appear in $P_f$ at the same starting index.
\end{lemma}
\begin{proof}
If the factor  $P_f[6\cdot 2^k : 6\cdot 2^k + n -1]$ first appears at index $6\cdot 2^k$ then the factor $P_f[6\cdot 2^k : 6\cdot 2^k + 2^k -1]$ must also first appear at index $6\cdot 2^k$ since $n \leq 2^k$. So, assume the factor $P_f[6\cdot 2^k : 6\cdot 2^k + n -1]$ first appears at index $4\cdot 2^k$. By lemma~\ref{pfapp}, this is the only other possible starting index. We create an automaton which accepts the pair $f$ and $n$ when $P_f[6\cdot 2^k : 6\cdot 2^k + n -1]$ first appears at index $4\cdot 2^k$ and $P_f[6\cdot 2^k : 6\cdot 2^k + 2^k -1]$ does not first appear at index $4\cdot 2^k$.
\begin{verbatim}
eval pftemp "?lsd_2 (n >= 7) & (Ex x >= 1 & $pfphi(x,n) & 
   (Ak (k<n) & PF[f][4*x+k] = PF[f][6*x+k]) & 
   (Er (r<x) & (PF[f][4*x+k] != PF[f][6*x+k])))":
\end{verbatim}
The automaton accepts no input showing that, when $P_f[6\cdot 2^k : 6\cdot 2^k + n -1]$ first appears at index $4\cdot 2^k$, then so does the factor  $P_f[6\cdot 2^k : 6\cdot 2^k + 2^k -1]$.
\end{proof}

\bigskip
Our main result is an exact formula for $S_f$ (and therefore $A_f$).

\bigskip

\begin{theorem}
\label{sf}
For $n \geq 7$,
$$
S_f (n) = 
\begin{cases}
   4 \cdot \phi(n) , & \text{if } \, \, \phi(n) = 2^k  \text{    and     }  f_{k+1} \neq f_{k+2}\\
   6 \cdot \phi(n) , & \text{if } \, \, \phi(n) = 2^k  \text{    and     }  f_{k+1} = f_{k+2} .
\end{cases}
$$
\end{theorem}
\begin{proof}
We start by showing that the theorem holds when $n$ is a power of $2$, so that $\phi(n) = n$. By lemmas \ref{pfapp} and \ref{last}, the factor $P_f[6\cdot \phi(n) : 6\cdot \phi(n) + n -1]$ is always the last length $n$ factor to appear in $P_f$ and appears first at either index $4\cdot \phi(n)$ or $6\cdot \phi(n)$. The following automaton takes as parallel input the folding instructions $f$ and an integer $n$ and accepts the input if $n$ is a power of $2$ and the last length $n$ factor to appear in $P_f$ (i.e. the factor $P_f[6\cdot \phi(n) : 6\cdot \phi(n) + n -1]$)  starts at index $4\cdot \phi(n)$.
\begin{verbatim}
eval pfpow24 "?lsd_2 (n >= 7) & $power2(n) & 
   (Ex x >= 1 & $pfphi(x,n) & $pffaceq(f,6*x, 4*x, n))":
\end{verbatim}

\begin{figure}[H]
\begin{center}
    \includegraphics[width=6in]{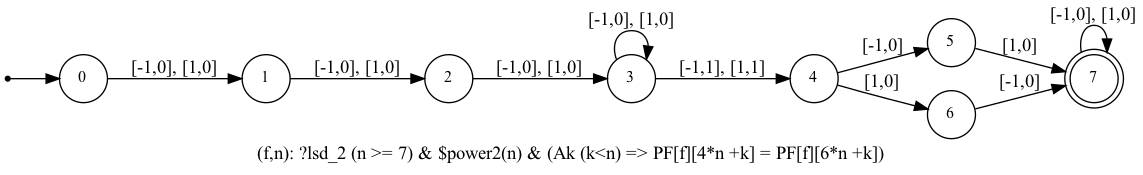}
    \end{center}
    \caption{Automaton {\tt pfpow24}.}
    \label{pfpow24fig}
\end{figure}

The automaton  {\tt pfpow24} is pictured in figure~\ref{pfpow24fig}. Remembering that the first element of the folding instructions $f$ has index $0$, it is clear from the picture that the pair $f$ and $n$ is accepted if and only if $n = 2^k$ for some $k \geq 3$ and $f_{k+1} \neq f_{k+2}$.

The next automaton takes as parallel input the folding instructions $f$ and an integer $n$ and accepts the input if $n$ is a power of $2$ and the last length $n$ factor to appear in $P_f$ starts at index $6\cdot \phi(n)$. 
\begin{verbatim}
eval pfpow26 "?lsd_2 (n >= 7) & $power2(n) & 
   (Ex, k ((x >=1) & (k<n)) & $pfphi(x,n) & 
   (PF[f][4*x+k] != PF[f][6*x+k]))":
\end{verbatim}

\begin{figure}[H]
\begin{center}
    \includegraphics[width=6in]{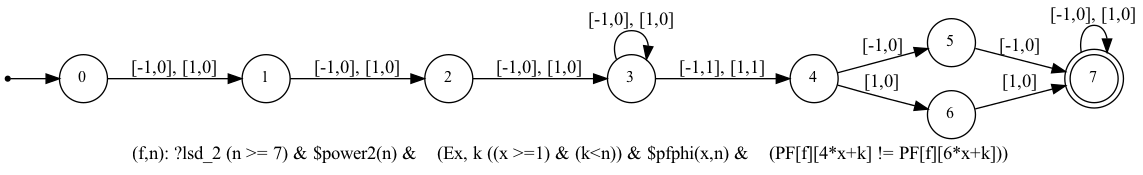}
    \end{center}
    \caption{Automaton {\tt pfpow26}.}
    \label{pfpow26fig}
\end{figure}

The automaton is pictured at figure~\ref{pfpow26fig}. It is clear from the picture that the pair $f$ and $n$ is accepted if and only if $n = 2^k$ for some $k \geq 3$ and $f_{k+1} = f_{k+2}$.

This completes the proof when $n$ is a power of $2$. The general case follows from lemma \ref{same}. If $n \geq 7$ and $\phi(n) = 2^k$, then the last length $n$ factor to appear is $P_f[6\cdot \phi(n) : 6\cdot \phi(n) + n -1]$. If  $\phi(n) = 2^k$, lemma \ref{same} says that this factor first appears at the same index as the factor $P_f[6\cdot 2^k : 6\cdot \phi(n) + 2^k -1]$. From above, this starting index is $4\cdot \phi(n)$ when $f_{k+1} \neq f_{k+2}$ and is $6\cdot \phi(n)$ when $f_{k+1} = f_{k+2}$.
\end{proof}

\bigskip

\begin{corollary}
\label{af}
For $n \geq 7$,
$$
A_f (n) = 
\begin{cases}
   4 \cdot 2^k + n - 1 , & \text{if } \, \, 2^{k-1} < n \leq 2^k  \text{    and     }  f_{k+1} \neq f_{k+2}\\
   6 \cdot 2^k + n - 1, & \text{if } \, \, 2^{k-1} < n \leq 2^k  \text{    and      }  f_{k+1} = f_{k+2} .
\end{cases}
$$
\end{corollary}

\bigskip

\begin{corollary}
$A_f(n) = 4\cdot \phi(n) + n - 1$ for all $n \geq 7$ if and only if 
$$
f = (f_0, f_1, f_2, f_3, 1, -1, 1, -1, \dots) \text{  or  }  f = (f_0, f_1, f_2, f_3, -1, 1, -1, 1, \dots)
$$ 
where $f_0, f_1, f_2, f_3 \in \{-1, 1 \}$. 

$A_f(n) = 6\cdot \phi(n) + n - 1$ for all $n \geq 7$ if and only if 
$$
f = (f_0, f_1, f_2, f_3, 1, 1, 1, \dots) \text{   or    } f = (f_0, f_1, f_2, f_3, -1, -1, -1, \dots)
$$ 
where $f_0, f_1, f_2, f_3 \in \{-1, 1 \}$.
\end{corollary}

\bigskip

\section{What happens when $n < 7$?}
When the word length is less than $7$, the appearance function displays a number of different behaviours. A calculation shows that, for folding instructions $f$,
\begin{align*}
S_f(1) \in \{ 2, 3 \} \,\,\, :&  \,\,\, A_f(1) \in \{ 2, 3 \} \\
S_f(2) \in \{ 4, 5, 6 \} \,\,\, :&  \,\,\, A_f(2) \in \{ 5, 6, 7 \} \\
S_f(3) \in \{ 14, 16, 22, 24 \} \,\,\, :&  \,\,\, A_f(3) \in \{ 16, 18, 24, 26 \} \\
S_f(4) \in \{ 14, 16, 22, 24 \} \,\,\, :&  \,\,\, A_f(4) \in \{ 17, 19, 25, 27 \} \\
S_f(5) \in \{ 28, 32, 44, 48 \} \,\,\, :&  \,\,\, A_f(5) \in \{ 32, 36, 48, 52 \} \\
S_f(6) \in \{ 31, 32, 47, 48 \} \,\,\, :&  \,\,\, A_f(6) \in \{ 36, 37, 52, 53 \}.
\end{align*}

To see how the choice of folding sequence $f$ determines $S_f(n)$ and $A_f(n)$, we use the automaton
\begin{verbatim}
eval pftemp# "?lsd_2 Ak (k < 50) => ($pfapp(f,r,#) & 
   (Es (s <=r) & (At (t<#) => PF[f][k+t] = PF[f][s+t])))":,
\end{verbatim}
replacing the symbol $\#$ with the integers $\{ 1, 2, 3, 4, 5, 6 \}$ as appropriate. The automaton accepts the pair $(f, r)$ when $S_f(\#) = r$. We can restrict the search to $k < 50$ because we know that, for each $f$, $S_f(n)$ is an increasing sequence and $S_f(7) = 48$.

We start with the case $n = 1$. The automaton
\begin{verbatim}
eval pftemp1 "?lsd_2 Ak (k < 50) => ($pfapp(f,r,1) & 
   (Es (s <=r) & (At (t<1) => PF[f][k+t] = PF[f][s+t])))":,
\end{verbatim}
is pictured in figure~\ref{sf1}.

\begin{figure}[H]
\begin{center}
    \includegraphics[width=6in]{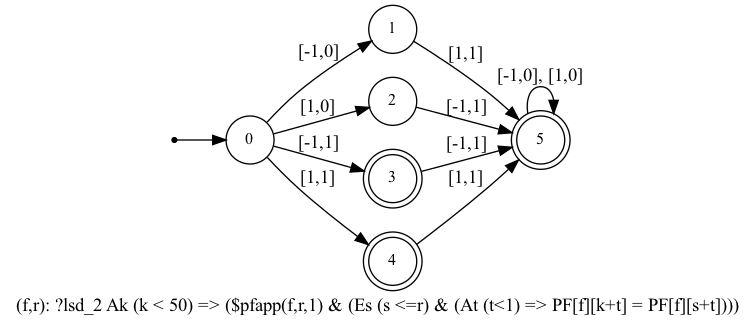}
    \end{center}
    \caption{Automaton for $S_f(1)$.}
    \label{sf1}
\end{figure}
State $5$ is the only accepting state because of the requirement that more than $\log_2 (k)$ folding instructions are read into the automaton. The automaton shows that 
\begin{align*}
S_f(1) = 2 \,\,\, &\text{when} \,\,\, f_0 \neq f_1 \\
S_f(1) = 3 \,\,\, &\text{when} \,\,\, f_0 = f_1 .
\end{align*}

The automaton for $S_f(2)$ is displayed in figure~\ref{sf2}. It shows that 
\begin{align*}
S_f(2) = 4 \,\,\, &\text{when} \,\,\, (f_0, f_1, f_2) \in \{ (-1, -1, 1), (-1, 1, -1), (1, -1, 1), (1, 1, -1) \} \\
S_f(2) = 5 \,\,\, &\text{when} \,\,\, (f_0, f_1, f_2) \in \{  (-1, 1, 1), (1, -1, -1) \}  \\
S_f(2) = 6 \,\,\, &\text{when} \,\,\, (f_0, f_1, f_2) \in \{ (-1, -1, -1), (1, 1, 1) \} .
\end{align*}

\begin{figure}[H]
\begin{center}
    \includegraphics[width=6in]{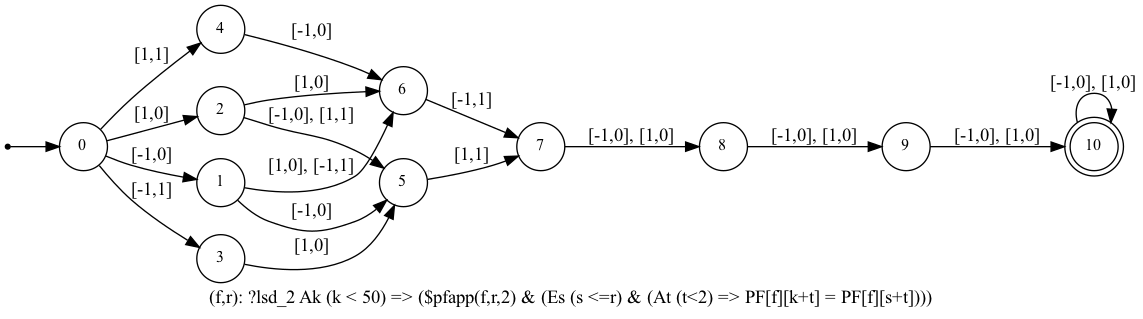}
    \end{center}
    \caption{Automaton for $S_f(2)$.}
    \label{sf2}
\end{figure}

Automata for $S_f(n)$ when $n \in \{3, 4, 5, 6 \}$ are pictured in figures~\ref{sf3}, \ref{sf4}, \ref{sf5} and \ref{sf6}. In summary, formulae can be derived for $S_f(3)$ and $S_f(4)$ in terms of $( f_1, f_2, f_3, f_4)$. A formula can be derived for $S_f(5)$ in terms of $( f_1, f_2, f_3, f_4, f_5)$. Finally, a formula can be derived for $S_f(6)$ in terms of $( f_0, f_1, f_2, f_3, f_4, f_5)$. 

\bigskip

\begin{figure}[H]
\begin{center}
    \includegraphics[width=6in]{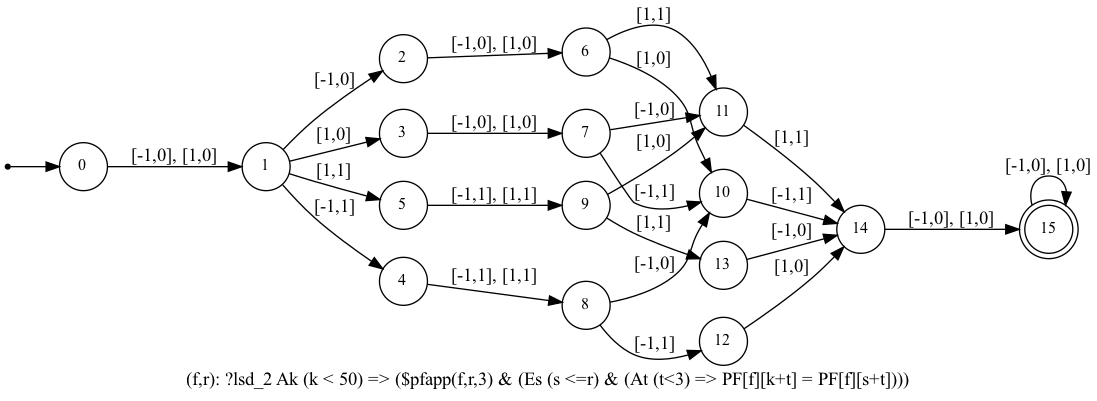}
    \end{center}
    \caption{Automaton for $S_f(3)$.}
    \label{sf3}
\end{figure}

\begin{figure}[H]
\begin{center}
    \includegraphics[width=6in]{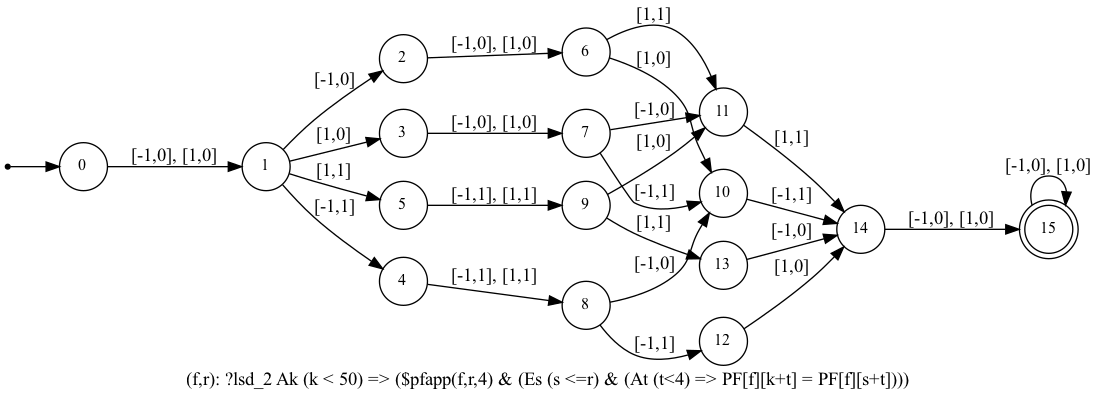}
    \end{center}
    \caption{Automaton for $S_f(4)$.}
    \label{sf4}
\end{figure}

\begin{figure}[H]
\begin{center}
    \includegraphics[width=6in]{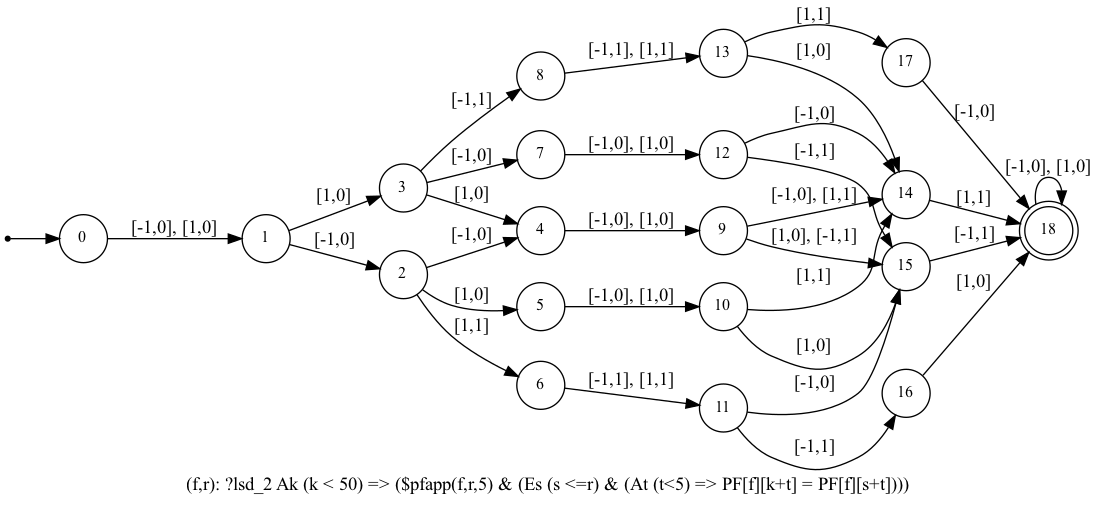}
    \end{center}
    \caption{Automaton for $S_f(5)$.}
    \label{sf5}
\end{figure}

\begin{figure}[H]
\begin{center}
    \includegraphics[width=6in]{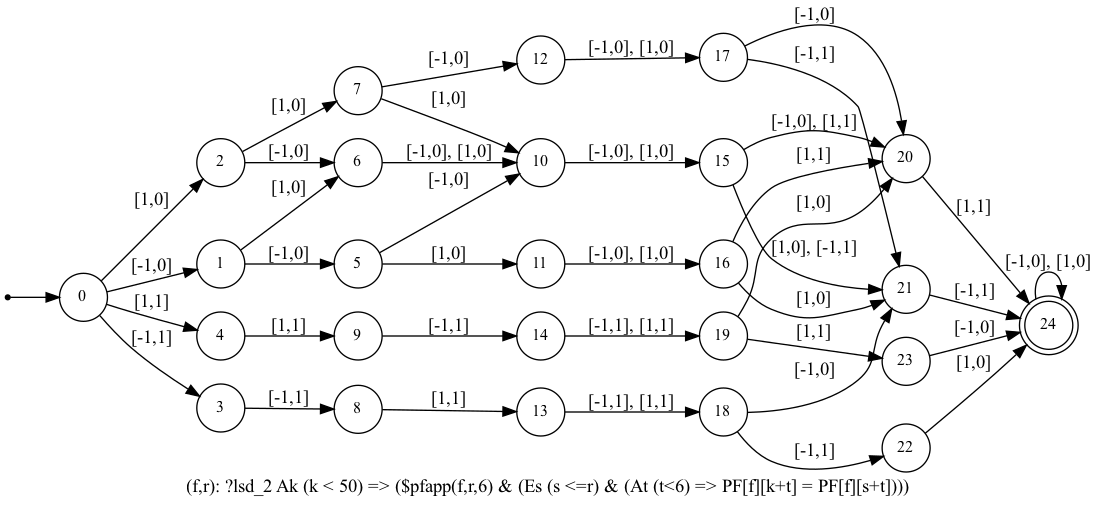}
    \end{center}
    \caption{Automaton for $S_f(6)$.}
    \label{sf6}
\end{figure}

\bigskip

Observe that the corresponding automaton for $S_f(7)$, which is shown in figure~\ref{sf7}, looks simple compared to that of smaller values of the factor length $n$.

\begin{figure}[H]
\begin{center}
    \includegraphics[width=6in]{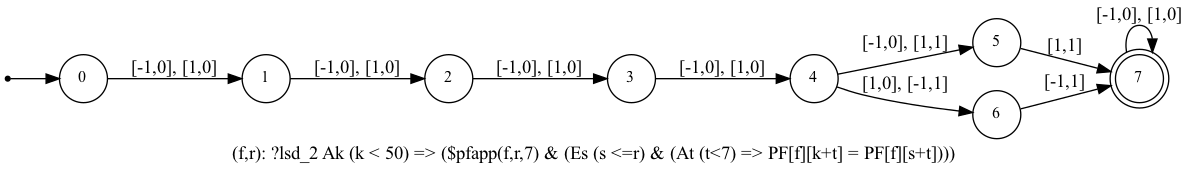}
    \end{center}
    \caption{Automaton for $S_f(7)$.}
    \label{sf7}
\end{figure}

\bigskip

\bibliographystyle{plain}
\begin{small}
\bibliography{Paperfolding}
\end{small}

\end{document}